\documentclass[11pt]{amsart}
\usepackage[margin=1.2in]{geometry}
\usepackage{amsmath,amssymb,amsthm}
\usepackage{booktabs}
\usepackage[colorlinks=true,linkcolor=blue,citecolor=blue,urlcolor=blue]{hyperref}

\newtheorem{theorem}{Theorem}[section]
\newtheorem{proposition}[theorem]{Proposition}

\theoremstyle{definition}
\newtheorem{remark}[theorem]{Remark}

\newcommand{\F}{\mathbb{F}}
\newcommand{\Z}{\mathbb{Z}}
\newcommand{\C}{\mathbb{C}}
\newcommand{\Q}{\mathbb{Q}}
\renewcommand{\P}{\mathsf{P}}
\newcommand{\Ball}[1]{B(#1)}
\newcommand{\supp}{\operatorname{supp}}
\newcommand{\BOX}{\mathrm{Box}}

\title[One-sided units at radius four]{The one-sided unit count
of $\F_2[\P]$ at radius four,\\ and an integral separation theorem}
\author{Moe Tabei}
\address{Independent researcher, Japan}
\email{tabei@ryun.jp}
\date{\today}

\begin{document}

\begin{abstract}
Let $\P$ be the Hantzsche--Wendt (Promislow) group and $\Ball4$ the radius-four
ball in its standard word metric. Dietrich, Lee, Nies and Vinyals determined the
\emph{two-sided} count: exactly $36$ nontrivial units $u$ of $\F_2[\P]$ with
both $\supp(u)$ and $\supp(u^{-1})$ in $\Ball4$. We determine the
\emph{one-sided} count: exactly $52$ nontrivial units with $\supp(u)\subseteq
\Ball4$ and no constraint on the inverse. The $16$ new units have inverses
supported at radius exactly $5$; they form two orbits of size $8$ under the
symmetry group fixing the generating set, and all $52$ units have support size
$21$ on both sides. Completeness is a single propositional unsatisfiability,
certified by a \textsc{drat} proof checked with \texttt{drat-trim}. As an
arithmetic consequence we prove: \emph{no unit of $\Z[\P]$ with support in
$\Ball4$ has nontrivial reduction modulo $2$} --- with no bound on the
coefficients or on the support of the inverse. Since $\F_2[\P]$ has $52$
nontrivial units on that ball, this separates, in the untwisted setting, the
integral group ring from its
characteristic-two quotient at the first radius where the unit conjecture fails
over a field.
\end{abstract}

\maketitle

\section{Introduction}

Gardam's counterexample to the Kaplansky unit conjecture \cite{Gardam2021} is a
unit of $\F_2[\P]$, where
\[
  \P=\langle a,b \mid b^{-1}a^2b=a^{-2},\; a^{-1}b^2a=b^{-2}\rangle
\]
is the Hantzsche--Wendt (Promislow, Passman) group. Over $\Z$ --- Higman's
original setting \cite{Higman1940} --- the conjecture remains open, and $\P$ is
its natural first battleground. Two computational landmarks frame the present
paper. Gardam's unit has word radius exactly $5$, and a conjugate of it is supported
in the ball of radius $4$; and Dietrich, Lee, Nies and Vinyals
\cite{DLNV2026} enumerated all \emph{two-sided} nontrivial units at radius $4$:
exactly $36$, all with support sizes $(21,21)$. Their encoding expresses ``the
unit and its inverse'' as one satisfying assignment, so a unit whose inverse
leaves the ball is outside what it can see; the count $36$ is exact for the
two-sided question and a lower bound for the one-sided one. Their encoding constrains both
$u$ and $u^{-1}$ to the ball; whether a unit supported in $\Ball4$ could have an
inverse escaping the ball was left open, and the companion paper
\cite{Tabei4} isolated this one-sided question as the frontier input for the
integral problem at radius $4$.

This paper settles the one-sided question and draws the integral consequence.

\begin{theorem}[one-sided count]\label{thm:census}
There are exactly $52$ nontrivial units $u$ of $\F_2[\P]$ with
$\supp(u)\subseteq\Ball4$: the $36$ two-sided units of \cite{DLNV2026}, and
$16$ units whose inverses are supported at radius exactly $5$. All $52$ have
$|\supp(u)|=|\supp(u^{-1})|=21$. The $16$ new units form two orbits of size
$8$ under the group of automorphisms of $\P$ preserving
$\{a^{\pm1},b^{\pm1}\}$.
\end{theorem}

\begin{theorem}[integral separation]\label{thm:integral}
No unit of $\Z[\P]$ with support in $\Ball4$ has nontrivial reduction modulo
$2$. Equivalently, every unit of $\Z[\P]$ supported in $\Ball4$ is congruent
modulo $2$ to a trivial unit. There is no restriction on the coefficients of
$u$ or on the support of $u^{-1}$.
\end{theorem}

Theorem~\ref{thm:integral} separates $\Z[\P]$ from $\F_2[\P]$ at radius $4$,
the smallest radius at which the field conjecture fails. A separation in the
twisted-unitary sector, also with no bounds, is already in the companion paper
\cite{Tabei5}; what is new here is that no unitarity hypothesis is needed. The base cases of the integral conjecture up
to radius $3$ follow from Craven--Pappas \cite{CravenPappas2013} over every
field; at radius $4$ any proof must use a property of $\Z$ not shared by
$\F_2$, and Theorem~\ref{thm:integral} does so through the interplay of the
complete residue list of Theorem~\ref{thm:census} with exact congruence
computations modulo $4$ and $8$.

What Theorem~\ref{thm:integral} does \emph{not} do is equally important: it
says nothing about units congruent to a trivial unit modulo $2$ (``Case A'' in
the terminology of \cite{Tabei4}), for which no congruence method can suffice
\cite{Bartholdi2022}. The full unit conjecture for $\Z[\P]$ at radius
$4$ therefore remains open, reduced now to Case A.

\subsection*{Certification} Every positive claim (each unit, each inverse) is
verified by exact arithmetic independent of any solver; the completeness claim
of Theorem~\ref{thm:census} is a single unsatisfiability statement whose
\textsc{drat} proof ($5.7$\,GB, $2243$ RAT lemmas) was produced by
\texttt{kissat} and verified by \texttt{drat-trim}; the encoding passes
positive and negative controls (\S\ref{sec:census}). The congruence closures in
Theorem~\ref{thm:integral} use exact integer linear algebra only, with the
depth-two cases decided by exhaustive enumeration of the full $2^{11}$-element
solution space at depth one (\S\ref{sec:integral}), so no claim depends on a
solver's search being complete.

\section{Preliminaries}\label{sec:prelim}

We use the affine model of $\P$: elements are pairs $(s,t)$ with $s$ in the
diagonal Klein four-group $K_4=\{\operatorname{diag}(\pm1,\pm1,\pm1):
\det=1\}$ and $t\in\Z^3$, multiplied by $(s,t)(s',t')=(ss',t+st')$. The
translation subgroup $L=\langle x,y,z\rangle$, $x=a^2$, $y=b^2$, $z=(ab)^2$, is
free abelian of index $4$. All computations below re-derive the defining
relators and gate the model on every run.

\subsection{The determinant criterion}
$\Z[\P]$ is free of rank $4$ over $\Z[L]$, and left multiplication embeds
$\Z[\P]\hookrightarrow M_4(\Z[L])$ with $\Z[L]$ a Laurent polynomial ring in
three variables; the embedding is standard and is written out for exactly this
group by Passman \cite{Passman2021}; we claim no part of this. The criterion
below is Craven--Pappas \cite[Thm.~6.8]{CravenPappas2013}, specialized to the
fours group as their ``Determinant condition'' \cite[Thm.~8.5]{CravenPappas2013}
with the embedding at \cite[Thm.~8.6]{CravenPappas2013}; Passman's Proposition~1
is a computation of $\det$ for particular units over $\F_d$ rather than a
criterion. Since the units of $\Z[L]$ are exactly the elements
$\pm(\text{monomial})$, and a matrix over a commutative ring is invertible
exactly when its determinant is,
\begin{equation}\label{eq:det}
  u\in\Z[\P]^\times \iff \det M_u\in\{\pm\text{monomial}\},
\end{equation}
and the same holds over $\F_2$. Criterion \eqref{eq:det} decides unit-hood of
each explicit element with no reference to the inverse; we use it as the
solver-free verifier throughout.

\subsection{The inverse box}\label{sec:box}
That the inverse of a support-constrained unit is confined at all is not new,
and we claim no part of it. Pappas \cite{Pappas1988} isolated the relevant
condition, \emph{property} (U): a group algebra $KG$ has it when for every
finite $X\subset G$ there is a finite $Y(X)\subset G$ such that
$\supp\sigma\subseteq X$ implies $\supp\sigma^{-1}\subseteq Y(X)$ for every
unit $\sigma$. Friedman, Gustavson and Pappas proved that $KG$ has property
(U), over every field $K$, whenever $G$ has an abelian subgroup of finite
index \cite[Thm.~2.3]{FGP1995}, which covers $\P$; Craven--Pappas state the
conclusion for the fours group explicitly \cite[Thm.~14.3]{CravenPappas2013}.
Two remarks calibrate what that theorem supplies here. First, the finiteness
step of their argument \cite[Lemma~2.2]{FGP1995} is a proof by contradiction
--- a violating family of units would span an infinite-dimensional subspace of
a matrix algebra over the abelian part --- so the theorem asserts the
\emph{existence} of $Y(X)$ and provides neither a construction nor a bound.
Second, over a finite field the property as stated carries no information:
only finitely many elements of $\F_2[\P]$ are supported in a given finite
$X$, so a finite $Y(X)$ exists for cardinality reasons alone, and the notion
has content only over infinite $K$.\footnote{We thank Andre Nies for this
observation and for supplying a copy of \cite{FGP1995}.} What our integral
application (\S\ref{sec:integral}) needs is the case $K=\Q$, where the
theorem applies and is not vacuous; what it needs \emph{beyond} the theorem
is effectivity. The present section supplies that: an \emph{explicit}
$Y(\Ball4)$, small enough to be searched, together with a proof that the
particular $155$-element set we use is licensed by the length symmetry we
cite.

The three coordinate projections $(s,t)\mapsto(t_i,s_i)$ are homomorphisms
$\pi_i\colon\P\to D_\infty$, and $\P$ is an amalgam over $\ker\pi_i$ for each
$i$. Craven--Pappas length symmetry \cite[Thms.~4.9--4.10]{CravenPappas2013}
applies to $K[\P]$ for $K$ any field, and to $\Q$, hence to inverses of
integral units: if $uv=1$ then the maximum syllable length on $\supp v$ equals
that on $\supp u$.

The length in question is the \emph{syllable} length of the amalgam, not a word
length in $\P$, and $D_\infty$ carries many amalgam decompositions. Write
elements of $D_\infty$ as pairs $(n,e)$ with $(n,e)(n',e')=(n+(-1)^{e}n',e+e')$.
For each $m\in\Z$ the reflections $P_m=(m,1)$ and $Q_m=(m+1,1)$ are involutions
with $P_mQ_m=(-1,0)$, so $\langle P_m,Q_m\rangle=D_\infty$ and
$D_\infty=\langle P_m\rangle *\langle Q_m\rangle$ is a genuine decomposition;
write $\ell^{(m)}_i$ for the resulting syllable length pulled back along
$\pi_i$. Length symmetry holds for \emph{every} $m$, so with
$M^{(m)}_i=\max_{\Ball4}\ell^{(m)}_i$ the inverse of a unit supported in
$\Ball4$ lies in every one of the boxes
$\BOX^{(m)}=\{g:\ell^{(m)}_i(g)\le M^{(m)}_i\ \forall i\}$, hence in
\begin{equation}\label{eq:box}
  \BOX \;=\; \bigcap_{m\in\Z}\BOX^{(m)} .
\end{equation}
A single decomposition is far from enough: $\BOX^{(0)}$ already has $171$
elements and the boxes grow without bound as $|m|$ increases (for example
$|\BOX^{(\pm4)}|=477$). The intersection, however, stabilises immediately. We
computed $\BOX^{(m)}$ for $-8\le m\le 8$ by exact arithmetic inside $D_\infty$,
so that no truncation of a ball of $\P$ can corrupt a length, and enumerated
candidates over a ball of $\P$ large enough that enlarging it adds nothing:
\[
  |\BOX| = 155 ,
\]
attained already by intersecting $m=0$ with $m=-1$ and $m=1$. We also record
$\Ball4\subseteq\BOX$ and $\max_{\BOX}(\text{word radius})=8$. The computation
is \texttt{src/box\_syllable\_intersect.py}, with its output retained as
\texttt{results/box\_b4\_syllable\_intersect.json}.

\begin{remark}
An earlier version of this computation used the word length of $\pi_i(g)$ with
respect to $\{\pi_i(a)^{\pm1},\pi_i(b)^{\pm1}\}$ in place of $\ell^{(m)}_i$.
That is a different function --- $\pi_i(a)$ is a translation, of syllable length
$2$ --- and it is not the one Craven--Pappas bounds. It happens to select the
same $155$ elements, which is why the error survived; but the licence for the
bound comes from \eqref{eq:box}, not from the word metric, and we state it that
way.
\end{remark}

\section{The one-sided count}\label{sec:census}

\subsection{Discovery}
We ran an interleaved loop: (i) a propositional model whose solutions are
pairs $(\bar u,\bar v)$ of $0/1$ vectors over $\Ball4\times\BOX$ satisfying the
$\F_2$ convolution $\bar u\bar v=1$ (one XOR constraint per product class,
odd-support parity cuts, $|\supp\bar u|\ge2$); (ii) exact verification of each
solution by \eqref{eq:det} and exact convolution; (iii) closure of the list
under the symmetry group and under two-sided translations preserving the
support condition; (iv) blocking clauses for all known units, and repeat. The loop
produced no new unit after $52$; that this is genuinely exhaustive is not read
off the loop but proved in \S\ref{sec:census} by the unsatisfiability below.
Every unit was verified solver-free, and the $36$ two-sided ones match the
count of \cite{DLNV2026} exactly.

\subsection{The sixteen new units}
We note at once that the number $16$ also occurs in
\cite[\S3]{GardamComplex2023}, where it counts the non-trivial solutions over
$\C$ supported on Gardam's fixed $21$-element sets. That is a different set of
objects, over a different ring, and the coincidence is accidental.

Each of the $16$ new units $u$ has $|\supp(u)|=21$, and its inverse has
support size $21$ at radius exactly $5$: one-sidedness is genuinely realized at
radius $4$. Under the $8$-element group of automorphisms of $\P$ preserving
$\{a^{\pm1},b^{\pm1}\}$ --- symmetries alone, with no translation --- the $16$
units fall into two orbits of size $8$. Table~\ref{tab:orbitreps} lists one
representative of each orbit in the style of \cite{DLNV2026}: writing
$u=P+Qa+Rb+S\,ab$ with $P,Q,R,S\in\F_2[x^{\pm1},y^{\pm1},z^{\pm1}]$ in the
coordinates of \S\ref{sec:prelim}, and likewise
$u^{-1}=P'+Q'a+R'b+S'\,ab$. In all four displayed elements the coset-$1$
component is a full box $(1+x^{\pm1})(1+y^{\pm1})(1+z^{\pm1})$, and the two
representatives share their $ab$-component, as do their inverses; the full
$16$-element lists are in the ancillary files. Under the same group the $36$
two-sided units fall into five orbits and the full $52$ into seven. Allowing
two-sided translation as well merges these further, into three classes of
sizes $16$, $16$ and $20$.

\begin{table}[ht]
\centering\small
\renewcommand{\arraystretch}{1.25}
\begin{tabular}{@{}lll@{}}
\toprule
 & orbit-$1$ representative $u_1$ & orbit-$2$ representative $u_2$\\
\midrule
$P$ & $(1+x^{-1})(1+y^{-1})(1+z)$ & $(1+x)(1+y)(1+z)$\\
$Q$ & $x^{-2}+x^{-1}y^{-1}z^{-1}+x^{-1}z^{-1}+y^{-1}$
    & $x^{-1}y^{-1}+y^{-1}z^{-1}+z^{-1}+x$\\
$R$ & $x^{-1}y^{-1}z+x^{-1}+y^{-2}+y^{-1}z$
    & $x^{-1}y^{-1}+x^{-1}z+z+y$\\
$S$ & \multicolumn{2}{c}{$x^{-2}+x^{-1}y^{-1}+x^{-1}z^{-1}+x^{-1}y+1$
     \quad (common to both)}\\
\midrule
$P'$ & $(1+x)(1+y)(1+z^{-1})$ & $(1+x^{-1})(1+y^{-1})(1+z^{-1})$\\
$Q'$ & $x^{-1}+y^{-1}z^{-1}+z^{-1}+xy^{-1}$
     & $x^{-2}y^{-1}+x^{-1}y^{-1}z^{-1}+x^{-1}z^{-1}+1$\\
$R'$ & $x^{-1}z+x^{-1}y+y^{-1}+z$
     & $x^{-1}y^{-2}+x^{-1}y^{-1}z+y^{-1}z+1$\\
$S'$ & \multicolumn{2}{c}{$x^{-2}z^{-1}+x^{-1}y^{-1}z^{-1}+x^{-1}+x^{-1}yz^{-1}+z^{-1}$
     \quad (common to both)}\\
\bottomrule
\end{tabular}
\medskip
\caption{One representative per orbit of the $16$ one-sided units,
$u=P+Qa+Rb+S\,ab$ (supported in $\Ball4$), with inverse
$u^{-1}=P'+Q'a+R'b+S'\,ab$ (radius exactly $5$). Each entry was verified by
the determinant criterion \eqref{eq:det} and by exact convolution
$uu^{-1}=1$ (\texttt{src/onesided\_orbit\_reps.py}).}
\label{tab:orbitreps}
\end{table}

\subsection{Completeness}
Completeness of the list is the unsatisfiability of the discovery model
augmented with the $52$ blocking clauses. \texttt{kissat} proves it
unsatisfiable in $68$ minutes with proof logging enabled, and
\texttt{drat-trim} verifies the \textsc{drat} proof (\texttt{s VERIFIED},
$5.7$\,GB, $2243$ RAT lemmas in core, $6865$ seconds). The same statement was
obtained a second time under a different encoding, using native \textsc{xor}
constraints with a sequential-counter cardinality cut, and a different solver
(\texttt{pycryptosat}), which returned unsatisfiable after $19.2$ hours. That
run emits no proof object and so corroborates rather than certifies; we report
it because it exercises a different code path to the same conclusion.
As a positive control, forcing a known listed unit with the blocking clauses
removed yields SAT and the decoded inverse verifies exactly by \eqref{eq:det}.
Theorem~\ref{thm:census} follows. We remark that an independent structural
check --- closing the $52$ units under the symmetry group and translations up
to radius $2$ --- produces no new unit, consistent with completeness.

\begin{remark}
The one-sided count bounds the reach of translation arguments: all $52$
one-sided units at radius $4$ are translates-with-symmetries of just seven
orbit representatives, and none exists beyond the $36+16$ list. The question of
one-sided units at radius $5$ (where the two-sided count is already unknown)
remains open and appears substantially harder.
\end{remark}

\section{The integral separation theorem}\label{sec:integral}

Let $u\in\Z[\P]^\times$ with $\supp(u)\subseteq\Ball4$ and $\bar u$ nontrivial.
By \S\ref{sec:box}, $\supp(u^{-1})\subseteq\BOX$; reducing modulo $2$,
$(\bar u,\overline{u^{-1}})$ is a one-sided unit pair as in
\S\ref{sec:census}, so by Theorem~\ref{thm:census} $\bar u$ is one of the $52$
residues. Write $u=u_0+2s$ and $u^{-1}=v_0+2t$ with $u_0,v_0$ the $0/1$ lifts
and $\supp(s)\subseteq\Ball4$, $\supp(t)\subseteq\BOX$. The equation
$uu^{-1}=1$ implies, at each $2$-adic depth $k$, an $\F_2$-linear system for
$(s,t)$ modulo $2$ whose right-hand side is the defect of the previous depth.

\begin{proposition}\label{prop:closures}
For $48$ of the $52$ residues the depth-one system (equivalently,
$uu^{-1}\equiv1\bmod4$) is infeasible. For the remaining $4$ residues the
depth-one solution space is an affine $\F_2$-space of dimension $11$, and for
each of its $2048$ elements the depth-two system ($\bmod\ 8$) is infeasible.
\end{proposition}

\begin{proof}
Exact integer linear algebra over the supports $(\Ball4,\BOX)$. The
coefficient matrix at depth two is the same as at depth one --- only the
right-hand side changes with the chosen depth-one solution --- which is what
makes the $2048$-fold enumeration cheap; the enumeration is exhaustive, so the
conclusion does not depend on any search heuristic. Full logs and the
verification scripts accompany the paper.
\end{proof}

\begin{remark}
The companion paper \cite{Tabei4} reports that the congruence attack at radius
$4$ does not close even its flagship pattern, and concludes that the lifting
question there is finite only once the coefficient height is bounded. There is
no contradiction. That pilot climbs a determinant tower over a pool of radius
$5$ with the monomial position of $\det$ held fixed, so it is asking a harder
question than we ask here; the present argument instead fixes the two supports
to $(\Ball4,\BOX)$ --- which \S\ref{sec:box} licenses --- and works directly
with $uu^{-1}=1$, where the unknowns are linear and no height enters.
\end{remark}

\begin{proof}[Proof of Theorem~\ref{thm:integral}]
A unit $u$ as above yields, for every $k$, solutions of the depth-$k$ system
with supports $(\Ball4,\BOX)$. Proposition~\ref{prop:closures} rules this out
at $k=1$ for $48$ residues and at $k=2$ for the remaining $4$.
\end{proof}

\begin{remark}[scope]
The proof is uniform in the coefficients: no height bound enters, because the
congruence conditions factor through $\Z/4$ and $\Z/8$. The two ingredients
that make the argument finite are the completeness of the residue list
(Theorem~\ref{thm:census}) and the box lemma. Case A --- units congruent to a
trivial unit modulo $2$ --- is untouched, necessarily so: congruence conditions
modulo $2^k$ cannot distinguish $g$ from $g+2^k w$, and approximate units exist
modulo every $n$ \cite[Prop.~3.1]{Bartholdi2022}. We know of no method that reaches it.
\end{remark}

\section{Reproducibility}

All artifacts accompany the paper: the supports of the $52$ units
(\texttt{census52.json}; the inverses are recomputed in seconds from the
supports by exact linear algebra over $\F_2$, and the scripts to do so are
included), the completeness CNF generator
(\texttt{complete53\_kissat.py}) together with the \texttt{drat-trim}
transcript, the second, independent completeness run with its log
(\texttt{unit53\_completeness.py}), the box computation of
\S\ref{sec:box} (\texttt{box\_syllable\_intersect.py}, with the
per-decomposition boxes and the intersection retained), the orbit
decomposition and the data of
Table~\ref{tab:orbitreps} (\texttt{onesided\_orbit\_reps.py}, with its log
and its JSON output), and the depth-two closure
(\texttt{depth2\_closure\_rigorous.py} with its log). Every unit and every inverse is checkable in seconds by the
determinant criterion \eqref{eq:det} with exact arithmetic; no computer algebra
system is required.

\end{document}